\documentclass[12pt,twoside]{amsart} 
 \title[On the minimal model theory in the case $\nu=0$]{On the minimal model theory for dlt pairs of numerical log Kodaira dimension zero}
\author{Yoshinori  Gongyo}
\address{Graduate School of Mathematical Sciences, 
the University of Tokyo, 3-8-1 Komaba, Meguro-ku, Tokyo 153-8914, Japan.}
\email{gongyo@ms.u-tokyo.ac.jp}
\date{2011/5/24, version 6.04}

\newcommand{\Supp}[0]{{\operatorname{Supp}}}
\newcommand{\Bs}[0]{{\operatorname{Bs}}}

\usepackage{latexsym}
\usepackage{amsmath}
\usepackage{amssymb}
\usepackage{amsthm}
\usepackage{amscd}
\usepackage{enumerate} 
\usepackage{amssymb} 
\usepackage{mathrsfs}          
\usepackage[all]{xy}

\usepackage{latexsym}

\newtheorem{thm}{Theorem}[section]

\newtheorem{prop}[thm]{Proposition}

\newtheorem{lem}[thm]{Lemma}

\newtheorem{cor}[thm]{Corollary}

\newtheorem{conj}[thm]{Conjecture}

\newtheorem{cl}[thm]{Claim}
 
\theoremstyle{definition}

\newtheorem{defi}[thm]{Definition}

\newtheorem{notanddef}[thm]{Notation and Definition}

\newtheorem{rem}[thm]{Remark}

\newtheorem*{ack}{Acknowledgments}

\begin{document}
\bibliographystyle{amsalpha+}
 
 \maketitle

\begin{abstract}
We prove the existence of good log minimal models for dlt pairs of numerical log Kodaira dimension $0$. 
\end{abstract}

\tableofcontents

\section{Introduction}Throughout this article, we work over $\mathbb{C}$, the complex number field. We will make use of the standard notation and definitions as in \cite{km} and \cite{kmm}. The minimal model conjecture for smooth varieties is the following:

\begin{conj}[Minimal model conjecture]\label{mmconj}Let $X$ be a smooth projective variety. Then there exists a minimal model or a Mori fiber space of $X$.

\end{conj}

This conjecture is true in dimension $3$ and $4$ by Kawamata, Koll\'ar, Mori, Shokurov and Reid (cf. \cite{kmm}, \cite{km} and \cite{sho2}). In the case where $K_X$ is numerically equivalent to some effective divisor in dimension $5$, this conjecture is proved by Birkar (cf. \cite{b1}). When $X$ is of general type or $K_X$ is not pseudo-effective, Birkar, Cascini, Hacon and $\mathrm{M^{c}}$Kernan prove Conjecture \ref{mmconj} in arbitrary dimension (\cite{BCHM}). Moreover if $X$ has maximal Albanese dimension, Conjecture \ref{mmconj} is true by \cite{F-alb}. In this article, among other things, we consider Conjecture \ref{mmconj} in the case $\nu(K_X)=0$ (for the definition of $\nu$, see Definition \ref{numerical Kodaira dimension}):

\begin{thm}\label{main1}Let $(X,\Delta)$ be a projective $\mathbb{Q}$-factorial dlt pair such that $\nu(K_X+\Delta)=0$. Then there exists a minimal model $(X_m,\Delta_m)$ of $(X,\Delta)$. 

\end{thm}
Actually the author heared from Vladimir Lazi\'c that Theorem \ref{main1} for klt pairs has been proved by Druel (cf. \cite{d} ) after he finished  this work. In this article, moreover, we give the generalization of his result for dlt pairs by using the sophisticated Birkar--Cascini--Hacon--$\mathrm{M^{c}}$Kernan's results and Druel's method. Essentially our method seems to be same as  Druel's. However, by expanding this result to dlt pairs we give the different proof of the abundance theorem for log canonical pairs in the case $\nu=0$ as shown by  \cite{ckp} and \cite{k-ab}:

\begin{thm}[=Theorem \ref{nonvani_lc}]\label{main2} Let $X$ be a normal projective variety  and $\Delta$ an effective $\mathbb{Q}$-divisor.
Suppose that $(X, \Delta)$ is a log canonial pair such that $\nu(K_X+\Delta)=0$. Then $K_X+\Delta$ is abundant, i.e. $\nu(K_X+\Delta)=\kappa(K_X+\Delta)$.
\end{thm}

For the first time, Nakayama proved Theorem \ref{main2} when $(X,\Delta)$ is klt. Nakayama's proof is independent of Simpson's results \cite{simp}. Simpson's results are used to approach the abundance conjecture in \cite{cpt}.   Campana--Perternell--Toma prove Theorem \ref{main2} when $X$ is smooth and $\Delta=0$. Siu also gave an analytic proof of it (cf. \cite{siu}). The results of \cite{ckp}, \cite{k-ab} and \cite{siu} depend on \cite{simp} and \cite{bu}. In this article, we show Theorem \ref{main2} by using a method  different from \cite{cpt}, \cite{ckp}, \cite{k-ab} and \cite{siu}. In our proof of Theorem \ref{main2}, we do not need results of \cite{simp} and \cite{bu}. Our proof depends on \cite{BCHM} and \cite{g}.

We summarize the contents of this article. 
In Section \ref{section2}, we define the {\em Kodaira dimension} and the {\em numerical Kodaira dimension}, and collect some properties of these. In Section \ref{section3}, we review a {\em log minimal model program with scaling} and works of Birkar--Cascini--Hacon--$\mathrm{M^{c}}$Kernan. In Section \ref{section4}, we introduce the {\em divisorial Zariski decomposition} and collect some properties of it. Section \ref{section5} is devoted to
the proof of the existence of minimal models in the case $\nu=0$. In Section \ref{section6}, we prove Theorem \ref{main2}.

\begin{notanddef} 
Let $\mathbb K$ be the real number field $\mathbb R$ or the rational 
number field $\mathbb Q$. We set $\mathbb{K}_{>0}=\{x \in \mathbb{K}| x>0\}$.

Let $X$ be a normal variety and let $\Delta$ be an effective $\mathbb K$-divisor 
such that $K_X+\Delta$ is $\mathbb K$-Cartier. 
Then we can define the {\em{discrepancy}} $a(E, X, \Delta)\in \mathbb K$ for every prime 
divisor $E$ {\em{over}} $X$. 
If $a(E, X, \Delta) \geq -1$ (resp.~$>-1$) 
for every $E$, then $(X, \Delta)$ is called {\em{log canonical}} 
(resp.~{\em{kawamata log terminal}}). 
We sometimes abbreviate log 
canonical (resp.~kawamata log terminal) to {\em{lc}} (resp.~{\em{klt}}). 

Assume that $(X, \Delta)$ is log canonical. 
If $E$ is a prime divisor over $X$ 
such that $a(E, X, \Delta)=-1$, then $c_X (E)$ is called a 
{\em{log canonical center}} ({\em{lc center}}, for short) of $(X, \Delta)$, where $c_X(E)$ 
is the closure of the image of $E$ on $X$. 
For the basic properties of log canonical centers, 
see \cite[Section 9]{F-fund}. 

Let $\pi:X \to S$ be a projective morphism of normal quasi-projective varieties and $D$ a $\mathbb{Z}$-Cartier divisor on $X$.  We set the complete linear system $|D/S| = \{E|D \sim_{\mathbb{Z},S} E \geq 0 \}$ of $D$ over $S$. The base locus of the linear system $|D/S|$ is denoted by 
$\Bs |D/S|$. When $S=\mathrm{Spec}\,\mathbb{C}$, we denote simply $|D|$ and $\Bs|D|$.

\end{notanddef}

\begin{ack}The author wishes to express his deep gratitude to Professor Osamu Fujino for various comments and discussions. He would like to thank his supervisor Professor Hiromichi Takagi for carefully reading 
a preliminary version of this article and valuable comments. He also wishes to thank Professor Caucher Birkar, Professor Noboru Nakayama and Professor F\'ed\'eric Campana for valuable comments. Moreover he would like to thank Professor Vyacheslav V. Shokurov for serious remarks about the Kodaira dimensions of $\mathbb{R}$-divisors and Doctor Vladimir Lazi\'c for teaching him \cite{d}. He is partially supported by the Research Fellowships of the Japan Society for the Promotion of Science for Young Scientists (22$\cdot$7399). Lastly, he would like to thank Professor St\'ephane Druel, Professor Nobuo Hara, Professor Yujiro Kawamata, Professor Yoichi Miyaoka, Professor Keiji Oguiso and Professor Mihai P{\u{a}}un for the encouragement and comments. He also thanks the referee for pointing out errors, useful comments, and suggestions.
\end{ack}

\section {Preliminaries}\label{section2}

\begin{defi}[Classical Iitaka dimension, cf. {\cite[II, 3.2, Definition]{N}}]\label{Kodaira dimension}
Let $X$ be a normal projective variety and $D$ an $\mathbb{R}$-Cartier divisor on $X$. If $| \llcorner mD \lrcorner| \not= \emptyset$, we put a dominant rational map
 $$\phi_{|\llcorner mD \lrcorner |} :X \dashrightarrow W_m
,$$
with respect to the complete linear system of $\llcorner mD \lrcorner$. We define the {\em Classical Iitaka dimension} $\kappa(D)$ of $D$ as the following:

$$ \kappa(D)=\mathrm{max} \{ \mathrm{dim} W_m \}
$$ 
 if  $H^0(X, \llcorner mD \lrcorner ) \not= 0$ for some positive integer $m$ and $\kappa(D)=- \infty$ otherwise.
\end{defi}

\begin{lem}\label{pro-kappa-dimension}Let $Y$ be a normal projective variety, $\varphi: Y \to X$ a projective birational morphism onto a normal projective variety, and let $D$ be an $\mathbb{R}$-Cartier divisor on $X$.  Then it holds the following:
\begin{itemize}
\item[(1)]  $\kappa(\varphi^*D)= \kappa(\varphi^*D +E)$ for any $\varphi$-exceptional effective $\mathbb{R}$-divisor $E$, and
\item[(2)]  $\kappa(\varphi^*D)= \kappa(D)$.
\end{itemize}
\end{lem}

\begin{proof} (1) and (2) follows from \cite[II, 3.11, Lemma]{N}.

\end{proof}

The following is remarked by Shokurov:

\begin{rem}\label{shokurov}In general, $\kappa(D)$ may not coincide with $\kappa(D')$ if $D \sim_{\mathbb{R}} D'$. For example, let $X$ be the $\mathbb{P}^1$, $P$ and $Q$ closed points in $X$ such that $P\not=Q$ and $a$ irrational number. Set $D=a(P-Q)$. Then $\kappa(D)=-\infty$ in spite of the fact that $D\sim_{\mathbb{R}}0$. 
\end{rem}

However, fortunately, $\kappa(D)$ coincides with $\kappa(D')$ if $D$ and $D'$ are effective divisors such that $D \sim_{\mathbb{R}} D'$ (\cite[Corollary 2.1.4]{choi}). Hence it seems reasonable that we define the following as the {\em Iitaka} ({\em Kodaira}) dimension for $\mathbb{R}$-divisors.

\begin{defi}[Invariant Iitaka dimension, {\cite[Definition 2.2.1]{choi}}, cf. {\cite[Section 7]{choisho}}]\label{new_Kodaira_dimension}Let $X$ be a normal projective variety and $D$ an $\mathbb{R}$-Cartier divisor on $X$.  We define the {\em invariant Iitaka dimension} $\kappa(D)$ of $D$ as the following:

$$ K(D)=\kappa(D')
$$ 
 if  there exists an effective divisor $D'$ such that $D \sim_{\mathbb{R}} D'$ and $K(D)=- \infty$ otherwise. Let $(X,\Delta)$ be a log canonical. Then we call $K(K_X+\Delta)$ the {\em log Kodaira dimension} of $(X,\Delta)$. 

\end{defi}

\begin{defi}[Numerical Iitaka dimension]\label{numerical Kodaira dimension} Let $X$ be a normal projective variety, $D$ an $\mathbb{R}$-Cartier divisor and $A$ an ample Cartier divisor on $X$. We set
$$\sigma(D,A)=\mathrm{max}\{k \in \mathbb{Z}_{\geq 0}| \limsup_{m \to \infty}m^{-k} \mathrm{dim} H^0(X, \llcorner mD \lrcorner+A) >0  \}
$$ 
 if  $H^0(X, \llcorner mD \lrcorner+A) \not= 0$ for infinitely many $m \in \mathbb{N}$ and $\sigma(D,A)=- \infty$ otherwise. We define 
 $$\nu(D)= \max \{ \sigma(D,A) | \text{$A$ is\ a\ ample\ divisor\ on\ $Y$} \}.
 $$
 Let $(X,\Delta)$ be a log pair. Then we call $\nu(K_X+\Delta)$ the {\em numerical log Kodaira dimension} of $(X,\Delta)$. If $\Delta=0$, we simply say $\nu(K_X)$ is the {\em numerical Kodaira dimension} of $X$.
 
\end{defi}

\begin{lem}\label{pro-nu-dimension}Let $Y$ be a normal projective variety, $\varphi: Y \to X$ a projective birational morphism onto a normal projective variety, and let $D$ be an $\mathbb{R}$-Cartier divisor on $X$.  Then it holds the following:
\begin{itemize}
\item[(1)]  $\nu(\varphi^*D)= \nu(\varphi^*D +E)$ for any $\varphi$-exceptional effective $\mathbb{R}$-divisor $E$, 
\item[(2)]  $\nu(\varphi^*D)= \nu(D)$, and
\item[(3)] $\nu(D)=\max\{k \in \mathbb{Z}_{\geq0}|D^{k} \not\equiv 0  \}$ when $D$ is nef.
\end{itemize}

\end{lem}

\begin{proof}See \cite[V, 2.7, Proposition]{N}.
\end{proof}

\begin{lem}[{\cite[V, 2.7, Proposition, (1)]{N}}]\label{invarinat} Let $X$ be a projective variety and $D$ and $D'$ $\mathbb{R}$-Cartier divisors on $X$ such that $D \equiv D'$. Then $\nu(D)=\nu(D')$. 

\end{lem}

\begin{rem}\label{nakayama} $\nu(D)$ is denoted as $\kappa_{\sigma}(D)$ in \cite[V, $\S$ 2]{N}. Moreover Nakayama also defined $\kappa_{\sigma}^-(D)$, $\kappa^+_{\sigma}(D)$ and $\kappa_{\nu}(D)$ as some numerical Iitaka dimensions. In this article we mainly treat in the case where $\nu(D)=0$, i.e. $\kappa_{\sigma}(D)=0$. Then it holds that  $\kappa_{\sigma}^-(D)=\kappa^+_{\sigma}(D)=\kappa_{\nu}(D)=0$ (cf. \cite[V, 2.7, Proposition (8)]{N}). Moreover, if  a log canonical pair $(X,\Delta)$ has a {\em weakly log canonical model} in the sense of Shokurov, then $\nu(K_X+\Delta)$ coincides with the {\em numerical log Kodaira dimension} in the sense of Shokurov by Lemma \ref{pro-nu-dimension} (cf. \cite[2.4, Proposition]{sho}). 
 
\end{rem}

\begin{defi}\label{base locus}Let $\pi:X \to S$ be a projective morphism of normal quasi-projective varieties and $D$ an $\mathbb{R}$-Cartier divisor on $X$. We set 
$$\mathbf{B}_{\equiv}(D/S)=\bigcap_{D \equiv_{S} E \geq 0} \mathrm{Supp}\,E.$$
When $S=\mathrm{Spec}\,\mathbb{C}$, we denote simply $\mathbf{B}_{\equiv}(D)$.
\end{defi}

We introduce a \emph{dlt blow-up}. The following theorem was originally proved by Professor 
Christopher Hacon (cf.~\cite[Theorem 10.4]{F-fund}, 
\cite[Theorem 3.1]{kk}). 
For a simpler proof, see \cite[Section 4]{F-ss}: 

\begin{thm}[Dlt blow-up]\label{dltblowup}
Let $X$ be a normal quasi-projective variety and 
$\Delta$ an effective $\mathbb{R}$-divisor on $X$ such 
that $K_X+\Delta$ is $\mathbb{R}$-Cartier. Suppose that $(X,\Delta)$ is log canonical.
Then there exists a projective birational 
morphism $\varphi:Y\to X$ from a normal quasi-projective 
variety with the following properties: 
\begin{itemize}
\item[(i)] $Y$ is $\mathbb Q$-factorial, 
\item[(ii)] $a(E, X, \Delta)= -1$ for every  
$\varphi$-exceptional divisor $E$ on $Y$, and
\item[(iii)] for $$
\Gamma=\varphi^{-1}_*\Delta+\sum _{E: {\text{$\varphi$-exceptional}}}E, 
$$ it holds that  $(Y, \Gamma)$ is dlt and $K_Y+\Gamma=\varphi^*(K_X+\Delta)$.
\end{itemize}

\end{thm}

The above theorem is very useful for studying log canonical singularities (cf. \cite{F-fund},  \cite{g}, \cite{kk}).\\

\section{Log minimal model program with scaling}\label{section3} In this section, we review a {\em log minimal model program with scaling} and introduce works by Birkar--Cascini--Hacon--$\mathrm{M^{c}}$Kernan.

\begin{lem}[cf. {\cite[Lemma 2.1]{b1}} and {\cite[Theorem 18.9]{F-fund}}]\label{scale}Let $\pi:X \to S$ be a projective morphism of normal quasi-projective varieties and $(X,\Delta)$ a $\mathbb{Q}$-factorial projective log canonical pair such that $\Delta$ is a $\mathbb{K}$-divisor. Let $H$ be an effective $\mathbb{Q}$-divisor such that $K_X+\Delta+H$ is $\pi$-nef and $(X,\Delta+H)$ is log canonical. Suppose that $K_X+\Delta$ is not $\pi$-nef. We put
$$\lambda=\inf\{\alpha \in \mathbb{R}_{\geq0}| K_X+\Delta+\alpha H\ \text{is\ $\pi$-nef} \}
.$$

Then  $\lambda \in \mathbb{K}_{>0}$ and there exists an extremal ray $R \subseteq \overline{NE}(X/S)$ such that $(K_X+\Delta).R <0$ and $(K_X+\Delta+\lambda H).R=0$.

\end{lem}

\begin{defi}[Log minimal model program with scaling]\label{LMMPS}Let $\pi:X \to S$ be a projective morphism of normal quasi-projective varieties and $(X,\Delta)$ a $\mathbb{Q}$-factorial projective divisorial log terminal pair such that $\Delta$ is a $\mathbb{K}$-divisor. Let $H$ be an effective $\mathbb{K}$-divisor such that $K_X+\Delta+H$ is $\pi$-nef and $(X,\Delta+H)$ is divisorial log terminal. We put
$$\lambda_1=\inf\{\alpha \in \mathbb{R}_{\geq0}| K_X+\Delta+\alpha H\ \text{is\ $\pi$-nef} \}
.$$ If $K_X+\Delta$ is not $\pi$-nef, then $\lambda_1>0.$
By Lemma \ref{scale}, there exists an extremal ray $R_1 \subseteq \overline{NE}(X/S)$ such that $(K_X+\Delta).R_1 <0$ and $(K_X+\Delta+\lambda_1 H).R_1=0$. We consider an extremal contraction with respect to this $R_1$. If it is a divisorial contraction or a flipping contraction, let 
$$(X,\Delta) \dashrightarrow (X_1,\Delta_1)$$
 be the divisorial contraction or its flip. Since $K_{X_1}+\Delta_1+\lambda_1H_1$ is $\pi$-nef, we put 
$$\lambda_2=\inf\{\alpha \in \mathbb{R}_{\geq0}| K_{X_1}+\Delta_1+\alpha H_1\ \text{is\ $\pi$-nef} \}
,$$
where $H_1$ is the strict transform of $H$ on $X_1$. Then we find an extremal ray $R_2$ by the same way as the above. We may repeat the process. We call this program a {\em log minimal model program with scaling of} $H$ over $S$. When this program runs as the following:
$$(X_0,\Delta_0)=(X,\Delta) \dashrightarrow (X_1,\Delta_1) \dashrightarrow \cdots \dashrightarrow (X_i, \Delta_i) \cdots, 
$$
then 
$$\lambda_1 \geq \lambda_2 \geq \lambda_3 \dots,
$$
where $\lambda_i=\inf\{\alpha \in \mathbb{R}_{\geq0}| K_{X_{i-1}}+\Delta_{i-1}+\alpha H_{i-1}\ \text{is\ $\pi$-nef} \}$ and $H_{i-1}$ is the strict transform of $H$ on $X_{i-1}$.
\end{defi}

The following theorems are slight generalizations of \cite[Corollary 1.4.1]{BCHM} and \cite[Corollary 1.4.2]{BCHM}. These seem to be  well-known for the experts.  

\begin{thm}[cf. {\cite[Corollary 1.4.1]{BCHM}}] Let $\pi:X \to S$ be a projective morphism of normal quasi-projective varieties and $(X,\Delta)$ be a $\mathbb{Q}$-factorial projective divisorial log terminal pair such that $\Delta$ is an $\mathbb{R}$-divisor. Suppose that $\varphi:X \to Y$ is a flipping contraction of $(X,\Delta)$. Then there exists the log flip of $\varphi$.

\end{thm}

\begin{proof} Since $-(K_X+\Delta)$ is $\varphi$-ample, so is $-(K_X+\Delta-\epsilon \llcorner \Delta \lrcorner)$ for a sufficiently small $\epsilon>0$. Because $\rho(X/Y)=1$, it holds that $K_X+\Delta \sim_{\mathbb{R},Y} c (K_X+\Delta-\epsilon \llcorner \Delta \lrcorner)$ for some positive number $c$. By \cite[Corollary 1.4.1]{BCHM}, there exists the log flip of $(X,\Delta-\epsilon \llcorner \Delta \lrcorner)$. This log flip is also the log flip of $(X,\Delta)$ since $K_X+\Delta \sim_{\mathbb{R},Y} c (K_X+\Delta-\epsilon \llcorner \Delta \lrcorner)$.

\end{proof}

\begin{thm}[cf. {\cite[Corollary 1.4.2]{BCHM}}]\label{bchm} Let $\pi:X \to S$ be a projective morphism of normal quasi-projective varieties and $(X,\Delta)$ be a $\mathbb{Q}$-factorial projective divisorial log terminal pair such that $\Delta$ is an $\mathbb{R}$-divisor. Suppose that there exists a $\pi$-ample $\mathbb{R}$-divisor $A$ on $X$ such that $\Delta \geq A$. Then any sequences of log minimal model program starting from $(X, \Delta)$ with scaling of $H$ over $S$ terminate, where $H$ satisfies that $(X,\Delta+H)$ is divisorial log terminal and $K_X+\Delta+H$ is $\pi$-nef.

\end{thm}

The above theorem is proved by the same argument as the proof of \cite[Corollary1.4.2]{BCHM} because \cite[Theorem E]{BCHM} holds on the above setting.  

\section{Divisorial Zariski decomposition}\label{section4} In this section, we introduce the {\em divisorial Zariski decomposition} for a pseudo-effective divisor.
\begin{defi}[cf. {\cite[III, 1.13, Definition]{N}} and \cite{kawamata_crepant}]\label{movile_limit} Let $\pi:X \to S$ be a projective morphism of normal quasi-projective varieties and $D$ an $\mathbb{R}$-Cartier divisor. We call that $D$ is a {\em limit of movable $\mathbb{R}$-divisors} over $S$ if $[D] \in \overline{\mathrm{Mov}}(X/S) \subseteq N^{1}(X/S)$ where $\overline{\mathrm{Mov}}(X/S)$ is the closure of the convex cone spanned by classes of fixed part free $\mathbb{Z}$-Cartier divisors over $S$.  When $S=\mathrm{Spec}\,\mathbb{C}$, we denote simply $\overline{\mathrm{Mov}}(X)$.

\end{defi}

\begin{defi}[cf. {\cite[III, 1.6, Definition and 1.12, Definition]{N}}]\label{Z-decomp}Let $X$ be a smooth projective variety and $B$ a big $\mathbb{R}$ divisor. We define 
$$ \sigma_{\Gamma}(B)=\inf\{\mathrm{mult}_{\Gamma}B'|B \equiv B'\geq0\}
$$
for a prime divisor $\Gamma$. Let $D$ be a pseudo-effective divisor. Then we define the following:

$$ \sigma_{\Gamma}(D)=\lim_{\epsilon \to 0 +}\sigma_{\Gamma} (D+\epsilon A)
$$
for some ample divisor $A$. We remark that $\sigma_{\Gamma}(D)$ is independent of the choice of $A$. Moreover the above two definitions coincide for a big divisor because a function $\sigma_{\Gamma}( \cdot )$
 on $\mathrm{Big}(X)$ is continuous where $\mathrm{Big}(X):=\{[B] \in \mathrm{N^1}(X)| \text{$B$\ is\ big}\}$ (cf. \cite[III, 1.7, Lemma]{N}). We set 
 $$N(D)=\sum_{\text{$\Gamma$:prime\ divisor}} \sigma_{\Gamma}(D) \Gamma\ \text{and}\ P(D)=D-N(D).
 $$
 We remark that $N(D)$ is a finite sum. We call the decomposition $D = P(D) + N(D)$ the {\em  divisorial Zariski decomposition} of $D$. We say that $P(D)$ (resp. $N(D)$) is the {\em positive part} (resp. {\em negative part}) of $D$. \end{defi}

Remark that the decomposition $D\equiv P(D)+N(D)$ is called several names: the {\em sectional decomposition} (\cite{kawamata_crepant}), the $\sigma$-{\em decomposition} (\cite{N}), the {\em divisorial Zariski decomposition} (\cite{bouck_dzd}), and the {\em numerical Zariski decomposition} (\cite{k-ab}).

\begin{prop}\label{nu=0}Let $X$ be a smooth projective variety and $D$ a pseudo-effective $\mathbb{R}$-divisor on $X$. Then it holds the following:

\begin{itemize}
\item[(1)] $\sigma_{\Gamma}(D)=\lim_{\epsilon \to 0+}\sigma_{\Gamma}(D+\epsilon E)$ for a pseudo-effective divisor $E$, and
\item[(2)] $\nu(D)=0$ if and only if $D \equiv N(D)$. 
\end{itemize}

\end{prop}

\begin{proof} (1) follows from \cite[III, 1.4, Lemma]{N}. (2) follows from \cite[V, 2.7, Proposition (8)]{N}.
\end{proof}

\section{Existence of minimal models in the case $\nu=0$}\label{section5}

\begin{thm}[cf. {\cite[Corollaire 3.4]{d}}]\label{termination} Let $X$ be a $\mathbb{Q}$-factorial projective variety and $\Delta$ an effective $\mathbb{R}$-divisor such that $(X,\Delta)$ is divisorial log terminal. Suppose that $\nu(K_X+\Delta)=0$. Then any log minimal model programs starting from $(X, \Delta)$ with scaling of $H$ terminate, where $H$ satisfies that $H\geq A$ for some effective $\mathbb{R}$-ample divisor $A$, $(X,\Delta+H)$ is divisorial log terminal, and $K_X+\Delta+H$ is nef.
\end{thm}

\begin{proof} Let $(X,\Delta) \dashrightarrow (X_1,\Delta_1)$ be a divisorial contraction or a log flip. Remark that it holds that
$$\nu(K_X+\Delta)=\nu(K_{X_1}+\Delta_1)$$
 from Lemma \ref{pro-nu-dimension} (1) and the negativity lemma. Now we run a log minimal model program 
 $$(X_i,\Delta_i) \dashrightarrow (X_{i+1},\Delta_{i+1})
$$
 starting from $ (X_0,\Delta_0)=(X, \Delta)$ with scaling of $H$. Assume by contradiction that this program does not terminate.
 Let $\{\lambda_i\}$ be as in Definition \ref{LMMPS}. We set  
$$\lambda=\lim_{i\to \infty}\lambda_i.
$$
If $\lambda \not =0$, the sequence is composed by $(K_X+\Delta+\frac{1}{2}\lambda H)$-log minimal model program. Thus the sequence terminates by Theorem \ref{bchm}. Therefore we see that $\lambda =0$. Now there exists $j$ such that $(X_i,\Delta_i) \dashrightarrow (X_{i+1},\Delta_{i+1})$
 is a log flip for any $i\geq j$. Replace  $(X,\Delta)$ by $(X_j,\Delta_j)$, we lose the fact that $A$ is ample. Then we see the following:

\begin{cl}\label{cl2}
$K_X+\Delta$ is a limit of movable $\mathbb{R}$-divisors.
\end{cl}

\begin{proof}[Proof of Claim \ref{cl2}]
See \cite[Step 2 of the proof of Theorem 1.5]{b2} or \cite[Theorem 2.3]{F-ss}.
\end{proof}

Let $\varphi: Y \to X$ be a log resolution of $(X,\Delta)$. We consider the divisorial Zariski decomposition 
$$\varphi^*(K_X+\Delta) = P(\varphi^*(K_X+\Delta))+N(\varphi^*(K_X+\Delta))$$
 (Definition \ref{Z-decomp}). Since 
$$\nu(\varphi^*(K_X+\Delta))=\nu(K_X+\Delta)=0,$$
 we see $P(\varphi^*(K_X+\Delta)) \equiv 0$ by Proposition \ref{nu=0} (2). Moreover we see the following claim:

\begin{cl}\label{cl1} 
$N(\varphi^*(K_X+\Delta))$ is a $\varphi$-exceptional divisor.
\end{cl}
\begin{proof}[Proof of Claim \ref{cl1}]
Let $G$ be an ample divisor on $X$ and $\epsilon$ a sufficiently small positive number. By Proposition \ref{nu=0} (1), it holds that
$$\mathrm{Supp}\,N(\varphi^*(K_X+\Delta)) \subseteq \mathrm{Supp}\,N(\varphi^*(K_X+\Delta+\epsilon G)).$$ 
If it holds that $\varphi_*(N(\varphi^*(K_X+\Delta))) \not =0$, we see that $\mathbf{B}_{\equiv}(K_X+\Delta+\epsilon G)$ has codimension $1$ components. This is a contradiction to Claim \ref{cl2}. Thus $N(\varphi^*(K_X+\Delta))$ is a $\varphi$-exceptional divisor.
\end{proof}
Hence $K_X+\Delta \equiv 0$, in particular, $K_X+\Delta$ is nef. This is a contradiction to the assumption.  
\end{proof}

\begin{cor}\label{mm} Let $X$ be a $\mathbb{Q}$-factorial projective variety and $\Delta$ an effective $\mathbb{R}$-divisor such that $(X,\Delta)$ is divisorial log terminal. Suppose that $\nu(K_X+\Delta)=0$. Then there exists a log minimal model of $(X,\Delta)$.

\end{cor}

\begin{rem}\label{rel_rem}These results are on the absolute setting. It may be difficult to extends these to the relative settings. See \cite{F-ss}. 

\end{rem}

\section{Abundance theorem in the case $\nu=0$}\label{section6} In this section, we prove the abundance theorem in the case where $\nu=0$ for an $\mathbb{R}$-divisor:

\begin{thm}\label{nonvani_lc} Let $X$ be a normal projective variety  and $\Delta$ an effective $\mathbb{R}$-divisor.
Suppose that $(X, \Delta)$ is a log canonical pair such that $\nu(K_X+\Delta)=0$. Then $\nu(K_X+\Delta)=K(K_X+\Delta)$. Moreover, if $\Delta$ is a $\mathbb{Q}$-divisor, then $\nu(K_X+\Delta)=\kappa(K_X+\Delta)=K(K_X+\Delta)$.
\end{thm}

First we extends \cite[Theorem 1.2]{g} to an $\mathbb{R}$-divisor.
 
\begin{lem}[cf. {\cite[Theorem 3.1]{fg}}]\label{g_K}Let $X$ be a normal projective variety  and $\Delta$ an effective $\mathbb{K}$-divisor.
Suppose that $(X, \Delta)$ is a log canonical pair such that $K_X+\Delta \equiv 0$. Then $K_X+\Delta \sim_{\mathbb{K}}0$

\end{lem}

\begin{proof} By taking a dlt blow-up (Theorem \ref{dltblowup}), we may assume that $(X,\Delta)$ is a $\mathbb{Q}$-factorial  dlt pair.  If $\mathbb K=\mathbb Q$, then 
the statement is nothing but \cite[Theorem 1.2]{g}. 
From now on, we assume that 
$\mathbb K=\mathbb R$. 
Let $\sum _i B_i$ be the irreducible decomposition of $\Supp\,\Delta$. 
We put $V=\underset{i}{\bigoplus}  \mathbb RB_i$. 
Then it is well known that 
$$
\mathcal L=\{ B\in V\, | \, (X, B)\ \text{is log canonical}\}
$$ 
is a rational polytope in $V$. 
We can also check that 
$$
\mathcal N=\{ B\in \mathcal L\, |\, K_X+B\ \text{is nef}\}
$$ 
is a rational polytope and $\Delta \in \mathcal N$ (cf. \cite[Proposition 3.2]{b2} and \cite[6.2 First Main theorem]{sho}). 
We note that $\mathcal N$ is known as Shokurov's polytope. 
Therefore, we can write 
$$
K_X+\Delta=\sum _{i=1}^k r_i (K_X+\Delta_i)
$$ 
such that 
\begin{itemize}
\item[(i)] $\Delta_i$ is an effective $\mathbb Q$-divisor such that $\Delta_i \in \mathcal{N}$ for 
every $i$, 
\item[(ii)] $(X, \Delta_i)$ is log canonical for every 
$i$, and 
\item[(iii)] $0<r_i<1$, $r_i\in \mathbb R$ for every $i$, and 
$\sum _{i=1}^k r_i =1$. 
\end{itemize}
Since $K_X+\Delta$ is numerically trivial 
and $K_X+\Delta_i$ is nef for every $i$, 
$K_X+\Delta_i$ is numerically trivial for every $i$. By \cite[Theorem 1.2]{g}, we see that $K_X+\Delta \sim_{\mathbb{R}}0$.

\end{proof}

\begin{proof}[Proof of Theorem \ref{nonvani_lc}] By taking a dlt blow-up (Theorem \ref{dltblowup}), we may assume that $(X,\Delta)$ is a $\mathbb{Q}$-factorial  dlt pair. By Corollary \ref{mm}, there exists a log minimal model $(X_m,\Delta_m)$ of $(X,\Delta)$. From Lemma \ref{pro-nu-dimension} (3), it holds that $K_{X_m}+\Delta_{m} \equiv 0$. By Lemma \ref{g_K}, it holds that $K(K_{X_m}+\Delta_m)=0$. Lemma \ref{pro-kappa-dimension} implies that $K(K_{X}+\Delta)=0$. If $\Delta$ is a $\mathbb{Q}$-divisor, then there exists an effective $\mathbb{Q}$-divisor $E$ such that $K_X+\Delta \sim_{\mathbb{Q}}E$ by Corollary \ref{mm} and Lemma \ref{g_K}. Thus we see that $\kappa(K_X+\Delta)=K(K_X+\Delta)$. We finish the proof of Theorem \ref{nonvani_lc}.

\end{proof}

\begin{cor}\label{rel-nonvani2}Let $\pi: X\to S$ be a projective surjective morphism of normal quasi-projective varieties, and let $(X,\Delta)$ be a projective log canonical pair such that $\Delta$ is an effective $\mathbb{K}$-divisor. Suppose that $\nu(K_F+\Delta_F)=0$ for a general fiber $F$, where $K_F+\Delta_F=(K_X+\Delta)|_{F}$. Then there exists an effective $\mathbb{K}$-divisor $D$ such that $K_X+\Delta \sim_{\mathbb{K}, \pi} D$. 
\end{cor}
\begin{proof}This follows from Theorem \ref{nonvani_lc} and \cite[Lemma 3.2.1]{BCHM}.
\end{proof}

\end{document}